\documentclass[12pt,oneside,reqno]{amsart}
\usepackage[utf8]{inputenc}
\usepackage[T1]{fontenc}
\usepackage[margin=1in]{geometry}
\usepackage{
amsmath,
amssymb,
amsthm,
amsrefs,
mathrsfs,
dsfont, 
xcolor,
graphicx,
float,
enumitem, 
comment,
caption, 
soul 
}
\definecolor{mygray}{gray}{0.3} 
\usepackage[colorlinks=true,urlcolor=gray,linkcolor=mygray,citecolor=mygray]{hyperref}
\newtheorem{theorem}{Theorem}
\newtheorem{proposition}[theorem]{Proposition}
\newtheorem{lemma}[theorem]{Lemma}

\newtheorem{assumption}{Assumption}

\newcommand{\eps}{\epsilon}

\newcommand{\la}{\lambda}

\newcommand{\om}{\omega}
\renewcommand{\phi}{\varphi}
\newcommand{\N}{\mathds{N}}

\newcommand{\R}{\mathds{R}}
\newcommand{\C}{\mathds{C}}
\newcommand{\D}{\mathds{D}}
\newcommand{\PP}{\mathbb{P}}
\newcommand{\dd}{\mathrm{d}}
\newcommand{\E}{\mathbb{E}}
\newcommand{\var}{\mathrm{Var}}
\newcommand{\CUE}{\mathrm{CUE}}
\newcommand{\OO}{\mathrm{O}}

\newcommand{\Sp}{\mathrm{Sp}}
\newcommand{\tr}{\mathrm{tr}}
\newcommand{\re}{\mathrm{e}}
\newcommand{\ri}{\mathrm{i}}
\newcommand{\ovr}{\overline}
\newcommand{\diag}{\mathrm{diag}}

\newcommand{\W}{\mathscr{W}}
\newcommand{\Wg}{\mathrm{Wg}}

\newcommand{\disteq}{\stackrel{d}{=}}

\title[Dynamics of a rank-one multiplicative
perturbation of a unitary matrix]{Dynamics of a rank-one multiplicative\\
perturbation of a unitary matrix}
\author{Guillaume Dubach}
\address{(GD) DMA, École Normale Supérieure -- PSL, 45 rue d'Ulm, F-75230 Cedex 5 Paris, France}
\email{guillaume.dubach@ens.fr}
\author{Jana Reker}
\address{(JR) Institute of Science and Technology Austria, 3400 Klosterneuburg, Austria}
\email{jana.reker@ista.ac.at}

\begin{document}

\begin{abstract}
We provide a dynamical study of a model of multiplicative perturbation of a unitary matrix introduced by Fyodorov. In particular, we identify a flow of deterministic domains that bound the spectrum with high probability, separating the outlier from the typical eigenvalues at all sub-critical timescales. These results are obtained under generic assumptions on $U$ that hold for a variety of unitary random matrix models.
\end{abstract}


\maketitle
\vspace{-.5cm}
\begin{figure}[h]
\begin{center}
\includegraphics[width=.3\textwidth]{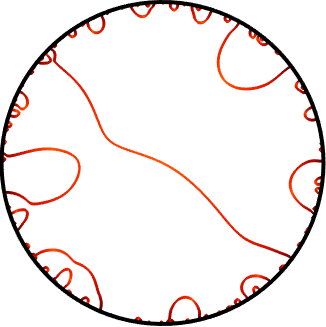} 
\quad
\includegraphics[width=.3\textwidth]{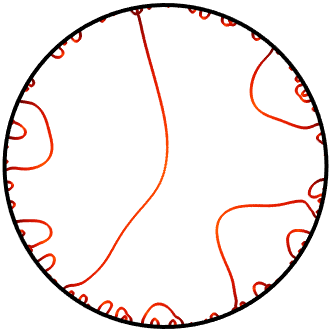} 
\quad
\includegraphics[width=.3\textwidth]{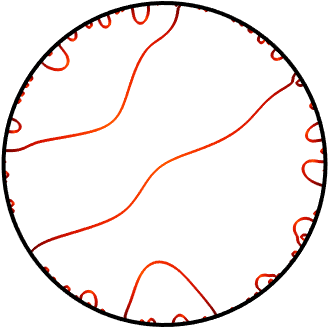}
\end{center}
\captionof{figure}{Trajectories of the eigenvalues of the $UA$ model of size $100\times100$. The evolution of the time parameter $t \in (-1,1)$ is represented by shades of red.}\label{fig1}
\end{figure}

\section*{Introduction}

The recent survey \cite{Forrester_Review} offers an overview of a variety of exact results concerning rank-one perturbations of random matrices. Among those are 
\begin{enumerate}[label=(\roman*)]
    \item Hermitian perturbations of Hermitian operators, where the now well-understood phenomenon of BBP transition occurs \cite{BBP};
    \item Generic perturbations of non-Hermitian operators, exemplified by the milestone work of Tao \cite{Tao2013};
\end{enumerate}
    as well as another, intermediate, class of models, namely,
\begin{enumerate}[label=(\roman*)]
    \item[(iii)] Perturbations of a normal (Hermitian or Unitary) operator resulting in a loss of normality.
\end{enumerate}
Interest in these particular models has its origin in physics, as they are relevant to scattering theory (see for instance the general introduction \cite{FyodorovSavin}), and as such, they have been the subject of several studies in recent years, both from the mathematical and the physics perspective. The model for which the most precise estimates are available to date is the rank-one anti-Hermitian perturbation of a Hermitian matrix, namely
\begin{equation}\label{GUE_model}
M(t) := H + \ri t vv^*, \qquad t \geq 0,
\end{equation}
where $H$ is a matrix with Gaussian entries independent up to the condition of Hermitianity (the so-called \textit{Gaussian Unitary Ensemble}, or GUE) and $v$ a unit vector, sampled uniformly on the sphere, and independent from $H$. This model was first studied in the pioneering works \cites{FyodorovSommers1996,FyodorovSommers1997,FyodorovKhoruzhenko}.
Recently, the more general case of a Wigner matrix was considered in \cite{DubachErdos}. It was shown that a unique outlier becomes separated for all times $\smash{t> 1+N^{-\frac13+ \epsilon}}$ for any $\epsilon>0$. Shortly thereafter, it was confirmed in \cite{FyodorovGUE}, using exact expressions available in the integrable GUE case, that $\smash{1+ \OO( N^{-\frac13})}$ is indeed the critical timescale for the emergence of the outlier. \\

Another relevant setting is the so-called $UA$ model, or rank-one multiplicative perturbation of a unitary matrix, first introduced by Fyodorov \cite{Fyodorov2001} (see also the review \cite{FyodorovSommers2003}). It is defined by
\begin{equation}\label{UA_model}
W(a) := U A, \qquad 
A = \diag(a, 1, \dots, 1), \quad
a \in [0,1],
\end{equation}
with the matrix $U$ being Haar-distributed on the unitary group (the so-called \emph{Circular Unitary Ensemble}, or CUE).
This model was studied in \cite{ForresterIpsen} at the timescale $a= \OO( N^{-\frac12})$, which happens to be the critical timescale for the emergence of an outlier at the origin.
The purpose of this paper is to conduct a dynamical study (replacing $a\in [0,1]$ by $t\in \R$, considered as a time parameter) of such $UA$ models under fairly generic assumptions on the matrix $U$.

\subsection*{Definitions and notations.}
Throughout the paper, $\tr(\cdot)$ denotes the (non-normalized) trace and $\Sp(\cdot)$ denotes the spectrum of a matrix. We denote the open unit disc and an open disk of center $z_0$ and radius $r$ by $\D$ and $D(z_0,r)$, respectively. The boundary of a set $S$ is denoted by $\partial S$ and its complement by $S^c$. \\

We say that an event (more precisely, a sequence of events) $\Omega=\Omega^{(N)}$ occurs \emph{with high probability (w.h.p.)} if 
\begin{equation}\label{whp}
\PP \left( \Omega^{(N)} \right) \xrightarrow[N \rightarrow \infty]{} 1,
\end{equation}
and that it occurs \emph{with overwhelming probability (w.o.p.)} if for any $A >0$,
\begin{equation}\label{wop}
    1-\PP \left( \Omega^{(N)} \right) < N^{-A}
\end{equation}
for large enough $N$. \\

An important tool that we need is the notion of \emph{uniform stochastic domination}, which is defined as follows. Let
\begin{displaymath}
X=\left\{X^{(N)}(z) \ : \ N\in\N, z\in \D \right\} 
\quad \text{and} \quad
Y=\left\{Y^{(N)}(z) \ : \ N\in\N, z \in \D\right\}
\end{displaymath}
be two families of complex random variables that are indexed by $N \in \N$ and $z \in \D$. We say that $X$ is \emph{stochastically dominated} by $Y$, uniformly for $z \in \mathscr{S}^{(N)} \subset \D$ if, for all $\eps>0$ we have that
$$
\forall z \in  \mathscr{S}^{(N)} \qquad |X^{(N)}(z)| \leq N^\epsilon |Y^{(N)}(z)|
$$
with overwhelming probability.
We denote this by $X\prec Y$, specifying over which domain $\mathscr{S}^{(N)}$ the uniformity holds. \\

We say that $G_N$ has a \emph{strongly separated outlier towards the origin for $t \in T \subset [-1,1]$} if there exist $\alpha_1 < \alpha_2 $, $c_1, c_2 \in (0,1)$ such that with high probability, for every $t \in T$,
\begin{enumerate}[label=(\roman*)]
    \item $D \left(0, c_1 \frac{N^{\alpha_1}}{1+N^{\alpha_1}}\right)$ contains exactly one eigenvalue of $G_N(t)$;
    \item $\D \backslash D \left(0, c_2 \frac{N^{\alpha_2}}{1+N^{\alpha_2}}\right)$ contains $N-1$ eigenvalues of $G_N(t)$.
\end{enumerate}
Whenever $\alpha_1<0$, we say that $G_N$ has a strongly separated outlier \emph{at the origin}.

\subsection*{Definition of the Model and statement of results}
Let $U\in \mathbb{U}_N(\C)$ be a unitary matrix and define
\begin{equation}\label{eq-defA}
A(t) :=I_N - (1-t) vv^*\in {M}_N(\C)
\end{equation}
with $t\in\R$ and a unit vector $v\in\C^N$. We consider the multiplicative rank one perturbation of $U$ that is given by 
\begin{equation}\label{def_UA}
G = G(t) = U A(t).
\end{equation} 
Note that this model coincides with \eqref{UA_model}
if $v=e_1$ and $|t|<1$, taking $a=t$. It is readily checked that $G(t)$ is unitary for $t=\pm 1$, so that its spectrum is in $\partial \D$. The eigenvalue trajectories of $G(t)$ for $|t|\leq1$ lie in $\overline{\D}$ by Weyl's inequality (as the smallest singular value of $G(t)$ is equal to $|t|$ while all others have modulus 1 by direct computation).
For $t=0$, the spectrum of $G(0)$ coincides with that of a truncated unitary matrix, with an extra eigenvalue at $0$. Observing that, for $t \neq 0$, $G$ is invertible with $G(t)^{-1}=G(t^{-1})^{*}$, the trajectories for $|t| \geq 1$ lie in the complement of $\D$, and their properties are readily obtained from the $|t| \leq 1$ case. We will hence focus our study of the spectrum of $G(t)$ to $t \in (-1,1)$, as represented on Figure \ref{fig1}. 

\medskip
For $U$ a unitary matrix and $v$ a unit vector, we define the quantity
\begin{equation}\label{def_om1}
\om_1^{(N)} := \sqrt{N} v^* U^* v
\end{equation}
and the functions
\begin{equation}\label{def_Ws}
\W(z) := v^* \frac{1}{I_N - z U^*} v,
\qquad
\W_2 (z) := v^* \frac{(zU^*)^2}{I_N - z U^*} v, 
\qquad
s_1(z) : = 1 + \frac{\om_1^{(N)}}{\sqrt{N}} z.
\end{equation}
Note that $\W$ is named by analogy with the `weighted resolvent' in \cite[Def.~2]{DubachErdos}, as it plays a similar role in the analysis below. However, there are important differences in the behavior and scaling of the two objects, which we address in more detail in the appendix. It is immediate to check that
\begin{equation}
\W(z) = s_1(z) + \W_2(z).
\end{equation}

Our main results are stated using the following assumptions.

\begin{assumption}[Simplicity]\label{bad_ass}
The spectrum of the random unitary matrix $U$ is almost surely simple; the unit vector $v$ is independent of $U$ and distributed with a non-vanishing density on the sphere $\mathbb{S}_{N-1}(\C)$.
\end{assumption}

Note that the simplicity of the spectrum holds for all typical models of random matrices. Whenever Assumption~\ref{bad_ass} is used, we denote the randomness in $v$ by $\PP_v$. In Section \ref{sec_Rouché}, we will work under the following regularity assumptions on $U$ and $v$.

\begin{assumption}[Regularity for $\W_2$]\label{smart_ass}
For any $\delta>0$,
\begin{equation}\label{assum_2.2_intro}
    \W_2(z) \prec \frac{|z|^2}{\sqrt{N} (1-|z|)}
\end{equation}
uniformly for $z \in \D_{\delta} := D(0,1-N^{-\delta})$.
\end{assumption}

\begin{assumption}[Regularity for $\om_1$]\label{kick_ass}
For any $\epsilon>0$,
\begin{equation}\label{assum_2.1_intro}
    N^{\eps} > |\om_1^{(N)}| > N^{-\epsilon}
\end{equation}
holds with high probability.
\end{assumption}

This is verified in particular for a Haar-distributed unitary matrix and a uniform unit vector $v$ independent from $U$, as will be proved in Theorem \ref{thm:assumption_holds}. Note that one of our results (Theorem \ref{main_theorem_precise_om1}) does not rely on Assumption \ref{kick_ass}, but only on the deterministic bound
\begin{equation}\label{om1_deterministic_bound}
|\om_1^{(N)}| \leq N^{\frac12},
\end{equation}
which follows from the definition of $\om_1^{(N)}$, and $U$ being unitary. \\

Our main result is a description of the different scales of these dynamics, under the regularity assumption. In particular, the emergence of an outlier at the origin is connected to the critical timescale $t = \OO(N^{-\frac12})$, in a way that can be described as follows:

\begin{theorem}\label{main_theorem_intro}
Under Assumptions \ref{smart_ass} and \ref{kick_ass}, for any $\alpha> \eps >0$, the following holds:
\begin{itemize}
\item There is a strongly separated outlier at the origin for $|t| < N^{-\frac12-\alpha}$.
\item With high probability, for all $|t| > N^{-\frac12 +\alpha}$, all eigenvalues are in $\D \backslash D(0,1-N^{-\alpha + \eps})$.
\end{itemize}
\end{theorem}

More precise statements are given as Theorems \ref{main_theorem_precise_om1} and \ref{main_theorem_precise_no_om1} in Section \ref{sec_Rouché}.
\smallskip

\subsection*{An overview of the method: Rouché's theorem and isotropic local law.}
The method used in \cite{DubachErdos} to study the model \eqref{GUE_model} essentially relies on two key facts:
\begin{enumerate}[label=(\roman*)]
    \item \label{kf1} The characterization of the spectrum at time $t$ as the preimage of $\ri/t$ by a random meromorphic function $\W$;
    \item \label{kf2} A uniform isotropic local law that allows to approximate $\W$ by a deterministic function on a contour.
\end{enumerate}
These facts together allow to locate the eigenvalues, and to isolate the outlier from the rest of the spectrum via Rouché's theorem, establishing that it is strongly separated\footnote{with the appropriate notion of `strong separation' in this setting.} at the proper timescales. The present case shares some important similarities. In particular, the analog of \ref{kf1} (stated below as Lemma \ref{lem-trajectories}) holds with the appropriate function $\W$, and some of the resulting properties, such as the almost sure non-intersection of trajectories, are analogous to the weakly non-Hermitian case as well. Apart from these formal analogies between the two models, the $UA$ model \eqref{def_UA} also involves some specific features. Especially, the isotropic local law \ref{kf2} is replaced by the regularity assumption on $\W_2$ (Assumption \ref{smart_ass}), which plays an equivalent role, but holds in fact more generally (see discussion in the Appendix). Another particular feature is the role of the parameter $\smash{\om_1^{(N)}}$ and the corresponding Assumption \ref{kick_ass}, which do not have an analog in the Hermitian setting.

\smallskip
\subsection*{Plan of the paper.}
Section \ref{sec_trajectories} contains our results concerning the trajectories of the $UA$ model \eqref{def_UA}.
We first derive general properties of the trajectories, only relying on Assumption~\ref{bad_ass}, in Section \ref{sec_prop_traj_simpl}, and then establish precise estimates on these trajectories under Assumptions~\ref{smart_ass} and \ref{kick_ass}, in Section \ref{sec_Rouché} -- from which Theorem \ref{main_theorem_intro} follows. Finally, in Section \ref{sec_CUE}, we focus on the particular CUE case, proving that it verifies all relevant assumptions, and also providing a simple proof that the timescales put forward in Theorem \ref{main_theorem_intro} are optimal. The relevance and generality of Assumption \ref{smart_ass} are discussed in more detail in the Appendix.
\bigskip

\section*{Acknowledgments}

The authors would like to thank Paul Bourgade, Djalil Chafaï, László Erd\H{o}s, Peter Forrester, Yan Fyodorov, and Boris Khoruzhenko for numerous discussions and comments. \\

G.D. gratefully acknowledges support from the Fondation des Sciences Mathématiques de Paris.
The work of J.R. is partially supported by ERC Advanced Grant `RMTBeyond' No.~101020331.

\section{Study of the trajectories}\label{sec_trajectories}

\subsection{Properties of the trajectories under the simplicity assumption.}\label{sec_prop_traj_simpl}
\begin{lemma}\label{lem-trajectories}
Let $z\in\D$ and consider the $UA$ model~\eqref{def_UA} with $t \in (-1,1)$. Then
\begin{displaymath}
z\in\Sp(G(t))\ \Leftrightarrow\ \W(z)=\frac{1}{1-t}.
\end{displaymath}
\end{lemma}
\begin{proof}
By definition, $z$ is in the spectrum of $G(t)$ if and only if
\begin{displaymath}
0 = \det(G(t)-z) = \det(U-z- (1-t) U vv^*).
\end{displaymath}
As $|z|<1$, this is equivalent to
\begin{displaymath}
\det(I_N - (1-t) (U-z)^{-1} U vv^*)=0.
\end{displaymath}
Using Sylvester's identity and defining $\W(z) = v^* (U-z)^{-1} U v$, we find that $z$ is in the spectrum of $G(t)$ if and only if
\begin{displaymath}
1-(1-t)\W(z) = 0
\end{displaymath}
which is the claim. \end{proof}
This strong property is due to the perturbation having rank one. In particular, it implies that the trajectories of the eigenvalues are given by the zero level lines of the imaginary part of $\W$, i.e.,
\begin{equation}
\bigcup_{j=1}^N\{\lambda_j(t):t\in (-1,1)\}=\{z\in\D: \Im\W(z)=0\}.
\end{equation}

Another important property that follows from the simplicity assumption is the non-intersection of the trajectories.

\begin{theorem}[Non-intersection of trajectories]\label{thm-nonintersect}
Under Assumption \ref{bad_ass}, for $N\geq 5$, the trajectories of the eigenvalues almost surely do not intersect, i.e.,
\begin{equation}
\PP \left( \left\{\exists s,t \in(-1,1), \ \exists i, j \in [\![1,N]\!] \ : \ (s,i) \neq (t,j), \ \lambda_i(s)=\lambda_j(t) \right\} \right)=0.
\end{equation}
\end{theorem}

\begin{proof}
The proof is analogous to that of~\cite[Thm.~2]{DubachErdos} with a few adjustments to account for the different geometry. We prove it for any fixed $U$ with a simple spectrum, only relying on the randomness of $v$ (i.e., according to the probability measure $\PP_v$). If $\lambda_i(s)=\lambda_j(t)$ for some $i,j\in [\![1,N]\!]$ and $s,t\geq(-1,1)$, Lemma~\ref{lem-trajectories} implies that
\begin{equation}\label{eq-nonintersect1}
\W(z)=\frac{1}{1-s}=\frac{1}{1-t}
\end{equation}
and thus $t=s$. So we only need to consider the possible intersection of two trajectories $\lambda_i$ and $\lambda_j$ ($i\neq j$) at the same time $t$. This implies that the function $\W(z)-\frac{1}{1-t}$ vanishes at $z\in\D$ with multiplicity at least two. Hence, any intersection point $z_0$ lies in the set
\begin{displaymath}
\mathcal{S}:=
\left\{z\in\D:\Im\W(z)=0\right\}
\cap
\left\{z\in\D:\W'(z)=0\right\}.
\end{displaymath}
We show that almost surely, there is no $z_0$ that satisfies both conditions. First, note that
\begin{displaymath}
\W'(z)=\sum_{j=1}^N|\langle u_j|v\rangle|^2\frac{\re^{-\ri\theta_j}}{(1-z\re^{-\ri\theta_j})^2},
\end{displaymath}
where $\re^{\ri\theta_1},\dots,\re^{\ri\theta_N}$ and $u_1,\dots,u_N$ denote the eigenvalues of $U$ and the associated unit eigenvectors, respectively. Hence, the condition $\W'(z_0)=0$ is equivalent to $z_0$ being the root of a complex polynomial of degree $2N-2$ ($\PP_v$-almost surely) that does not vanish on the unit circle. In particular, there is a finite number of such points. Allowing for multiplicity, we denote them by $(Z_k)_{k=1}^{2N-2}$.

\medskip
Next, note that $0\notin\mathcal{S}$ by inspection.
We may thus assume $z\neq0$ and replace the condition $\W'(z)=0$ by $z\W'(z)=0$. For any $z\notin\partial \D$, define the real vectors
\begin{displaymath}
Y_1(z):=\left(\Im \frac{1}{1-z\re^{-\ri\theta_j}}\right)_{j=1}^N,\quad Y_2(z):=\left(\Re \frac{z\re^{-\ri\theta_j}}{(1-z\re^{-\ri\theta_j})^2}\right)_{j=1}^N,\quad Y_3(z):=\left(\Im \frac{z\re^{-\ri\theta_j}}{(1-z\re^{-\ri\theta_j})^2}\right)_{j=1}^N
\end{displaymath}
as well as
\begin{displaymath}
X:=(|\langle u_j|v\rangle|^2)_{j=1}^N.
\end{displaymath}
With this notation, the condition $\Im\W(z)=0$ can be written as $\langle X|Y_1(z)\rangle$ while $z\W'(z)=0$ translates to $\langle X|Y_2(z)\rangle=\langle X|Y_3(z)\rangle=0$. The remainder of the proof relies on the following (deterministic) observation.

\begin{lemma}\label{lem-geometry}
Let $N\geq5$. For any $z \in \D \backslash \{0\}$, it holds that $Y_1(z)\notin\mathrm{Span}_\R(Y_2(z),Y_3(z))$.
\end{lemma}

\begin{proof}
For $z\neq0$ with $|z|<1$ denote
\begin{displaymath}
a_j+\ri b_j=\frac{1}{1-z\re^{-\ri\theta_j}}
\end{displaymath}
which can be rearranged to give
\begin{displaymath}
|z|=|z\re^{-\ri\theta_j}|=\left|1-\frac{1}{a_j+\ri b_j}\right|
\end{displaymath}
i.e., $(a_j+\ri b_j)^{-1}$ lies on a circle with center $1$ and radius $|z|$. Hence, by inversion, $a_j+\ri b_j$ lies on a circle with center $(1-|z|^2)^{-1}$ and radius $|z|(1-|z|^2)^{-1}$. We denote this circle by $\mathscr{C}_{|z|}$. As the phases $\theta_j$ are distinct by assumption, $(a_j,b_j)$ are $N$ distinct points on $\mathscr{C}_{|z|}$.

\medskip
Assume now that $Y_1(z)=\alpha Y_2(z)+\beta Y_2(z)$ for some $\alpha,\beta\in\R$. Since
\begin{displaymath}
\frac{z\re^{-\ri\theta_j}}{(1-z\re^{-\ri\theta_j)^2}}=\frac{1-(1-z\re^{-\ri\theta_j})}{(1-z\re^{-\ri\theta_j)^2}}=(a_j+\ri b_j)^2-(a_j+\ri b_j),
\end{displaymath}
and $Y_1(z)=b_j$, we can rewrite the linear dependence condition as
\begin{displaymath}
b_j=\alpha(a_j^2-b_j^2-a_j)+\beta(2a_jb_j-b_j).
\end{displaymath}
Note that this describes a hyperbola $\mathscr{H}_{\alpha,\beta}$ (in the degenerate case, we obtain a union of two lines). Bézout's theorem for curves gives a bound as to the number of intersection points between this curve and a circle:
\begin{displaymath}
N\leq|\mathscr{C}_{|z|}\cap\mathscr{H}_{\alpha_\beta}|\leq 4,
\end{displaymath}
which is the desired contradiction.
\end{proof}

Lemma~\ref{lem-geometry} implies that
\begin{displaymath}
\forall z\in\D,\quad \PP_v(\Im\W(z)=0 \ | \ \W'(z)=0)=0,
\end{displaymath}
as $Y_1$ is linearly independent of $Y_2(z),Y_3(z)$ and $X$ is distributed with a non-vanishing density on the simplex
\begin{displaymath}
\left\{(x_1,\dots,x_N)\ :\  x_j\geq0, \sum_{j=1}^Nx_j=1\right\}.
\end{displaymath}
by Assumption \ref{bad_ass}. We aim to deduce that
\begin{displaymath}
\PP_v(\exists z\in\D: \Im\W(z)=0\ \wedge\ \W'(z)=0)=0
\end{displaymath}
to conclude that trajectories originating from two distinct eigenvalues $\PP_v$-a.s. do not intersect. Let $\pi\in\mathfrak{S}_{2N}$ be a uniformly random permutation and let $\PP$ denote the probability with respect to $\pi$ as well as $v$. Moreover, define
\begin{displaymath}
N_z:=|\{k \ : \ Z_k=z\}|
\end{displaymath}
and recall that the points $Z_k$ were defined as the zeros of $\W'$, i.e., $\W'(z)=0$ if and only if $N_z\geq1$. The randomness in $\pi$ ensures that
\begin{displaymath}
\PP(Z_{\pi(k)} =z \ | \ N_z\geq1)=\sum_{m=1}^{2N-2}\PP(Z_{\pi(k)}=z \ | \ N_z=m)\geq\frac{1}{2N}
\end{displaymath}
so that we can write
\begin{align*}
0=\PP(\Im\W(z)=0 \ | \ N_z\geq1)&\geq \PP(\Im \W(z)=0 \ | \ Z_{\pi(k)}=z)\PP(Z_{\pi(k)}=z \ | \ N_z\geq1)\\
&\geq\frac{1}{2N}\PP(\Im\W(Z_{\pi(k)})=0 \ | \ Z_{\pi(k)}=z)
\end{align*}
for fixed $k=1,\dots,2N-2$. This implies $\PP(\Im \W(Z_k)=0 \ | \ Z_{\pi(k)}=z)=0$ and it follows that
\begin{displaymath}
\PP(\Im \W(Z_k)=0)=0
\end{displaymath}
by integrating out the conditioning. As there are only finitely many choices for $k$, a union bound yields
\begin{displaymath}
\PP_v(\exists k:\Im \W(Z_k)=0)=\PP(\exists k:\Im \W(Z_{\pi(k)})=0)\leq\sum_{k=1}^{2N-2}\PP(\Im \W(Z_{\pi(k)})=0)=0,
\end{displaymath}
implying $\PP_v(\mathcal{S})=0$. We conclude that
\begin{displaymath}
\PP_v(\exists t \in (-1,1), i\neq j: \lambda_j(t)=\lambda_j(t))=0
\end{displaymath}
which completes the proof.
\end{proof}

Non-intersection of trajectories is an essential input, as it allows to identify each $\la_k(t)$ as the unique continuous trajectory linking the eigenvalue $\re^{\ri \theta_k}$ at time $t=1$ to one of the eigenvalues at any time $t$. In particular, we may consider the trajectory of the eigenvalue that eventually becomes an outlier (for $t$ approaching $0$) and give meaning to concepts such as the `origin' of the outlier, the shape, and length of individual trajectories, etc. A realization with two trajectories getting very close to one another is given in Figure \ref{fig2}.\\

Remarkably, the evolution of the eigenvalues of the $UA$ model is described by a closed system of differential equations that governs the behavior of the trajectories. We provide these equations below for $t\in(0,1]$ (the equations for $t\in[-1,0)$ can be derived similarly). Note that this result is not used as input for our main theorems. However, it is an important feature of these dynamics, and may be of interest for future studies of this model.

\begin{theorem}[Differential Equations] Assume that $U$ satisfies Assumption~\ref{bad_ass}. We denote its eigenvalues by $\re^{\ri\theta_1},\dots,\re^{\ri\theta_N}$ with $\theta_j\in[0,2\pi)$ and the associated unit eigenvectors by $u_1,\dots,u_N$. Let $v$ be a unit vector and $t\in(0,1]$. Then the evolution of the eigenvalues $\lambda_1(t),\dots,\lambda_N(t)$ of $G(t)=U(I_N-(1-t)vv^*)$ is governed by the differential equations
\begin{equation}\label{eq-ODEs}
\lambda_j'(t)=\frac{1-|\lambda_j(t)|^2}{t(t^2-1)}
\left(\prod_{k=1}^N\lambda_k(t)\right)
\prod_{k\neq j}\frac{\lambda_j\overline{\lambda_k(t)}-1}{\lambda_j(t)-\lambda_k(t)},\quad j=1,\dots,N,
\end{equation}
with initial conditions $\lambda_j(1)=\re^{\ri\theta_j}$ and
\begin{equation}\label{eq-ODEinitial}
\lambda_j'(1)=\re^{\ri\theta_j}|\langle u_j|v\rangle|^2,\quad j=1,\dots,N.
\end{equation}
\end{theorem}

\begin{proof}
Let $L_1,\dots, L_N, R_1,\dots, R_N$ be a bi-orthogonal system of left and right eigenvectors associated with $G(t)$. We start with the basic perturbative identity
\begin{equation}\label{eq-pert_id}
\lambda_j'(t)=\langle L_j|G'(t)|R_j\rangle=\langle L_j|U|v\rangle\langle v|R_j\rangle
\end{equation}
to obtain a dynamic description of the eigenvalues (cf. proof of~\cite[Thm.~4]{DubachErdos}). Using that $\langle L_j|(UA)=\lambda_j\langle L_j|$ and that $A^{-1}=I-(1-t^{-1})vv^*$, the relation~\eqref{eq-pert_id} can be rewritten as
\begin{equation}\label{eq-pert_id2}
\lambda_j'(t)=\frac{\lambda_j(t)}{t}\langle L_j|v\rangle\langle v|R_j\rangle.
\end{equation}
For $t=1$, we have $G(t)=U$ and the eigenvalues are given by $\lambda_j(1)=\re^{\ri\theta_j}$. Moreover, $L_j=R_j=u_j$, which yields the initial condition stated in~\eqref{eq-ODEinitial}. For $t\in(0,1)$, note that the quantity $\langle L_j|v\rangle\langle v|R_j\rangle$ is invariant under a unitary change of basis. It is thus sufficient to compute it for a Schur form $T$ of $G$. W.l.o.g. let $\lambda_1,\dots,\lambda_N$ be the order in which the eigenvalues appear on the diagonal of $T$, i.e., $\lambda_j=T_{jj}$. By exchangeability, it is sufficient to show~\eqref{eq-ODEs} for $j=1$. Writing $G=STS^*$ with a suitable unitary matrix $S$, note that
\begin{displaymath}
TT^*=SGG^*S^*=SAA^*S=I-(t^2-1)Svv^*S^*
\end{displaymath}
i.e., $TT^*$ is a rank-one perturbation of identity. Let $\widetilde{v}:=Sv$. By \cite[Lem.~4]{Dubach_TUE}, we can express the columns of $T$ as well as the entries of $\sqrt{t^2-1}\widetilde{v}$ in terms of the eigenvalues $\lambda_1,\dots,\lambda_N$. In particular, the $k$-th column of $T$ is of the form 
\begin{displaymath}
(T_{kj})_{j=1}^N=\begin{pmatrix} \tau_{k}\\ \lambda_k\\0\end{pmatrix},\quad \tau_k=\frac{|\lambda_k|^2-1}{\overline{\lambda_k}}\frac{1}{\widetilde{v}_k}\widetilde{v}^{(k-1)}
\end{displaymath}
with $\widetilde{v}^{(k-1)}$ being the column vector containing the first $k-1$ entries of $\widetilde{v}$. As $T$ is (upper) triangular, we further have $R_1=e_1$ and thus $\langle \widetilde{v}|R_1\rangle=\overline{\widetilde{v}_1}$. For the following argument, denote $L_1=(1,\ell_2,\dots,\ell_N)$ and $L_1^{(d)}=(1,\ell_2,\dots,\ell_d)$ for $d=1,\dots,N$. This vector satisfies the recursion
\begin{equation}\label{eq-Lrecursion}
\ell_{k+1}=\frac{1}{\lambda_1-\lambda_{k+1}}\langle L_1^{(k)}|\tau_{k+1}\rangle=\frac{1}{\lambda_1-\lambda_{k+1}}\frac{|\lambda_{k+1}|^2-1}{\overline{\lambda_{k+1}}\widetilde{v}_{k+1}}\langle L_1^{(k)}|\widetilde{v}^{(k)}\rangle,
\end{equation}
where $\tau_{k+1}$ denotes the first $k$ entries of the $(k+1)$-th column of $T$. We use this connection to compute $\langle L_1|\widetilde{v}\rangle$ recursively. The initial condition is given by
\begin{displaymath}
\langle L_1^{(1)}|\widetilde{v}^{(1)}\rangle=\widetilde{v}_1
\end{displaymath}
and we iterate the recursion by
\begin{align*}
\langle L_1^{(k+1)}|\widetilde{v}^{(k+1)}\rangle&=\langle L_1^{(k)}|\widetilde{v}^{(k)}\rangle+\ell_{k+1}\widetilde{v}_{k+1}\\
&=\left(1+\frac{1}{\lambda_1-\lambda_{k+1}}\frac{|\lambda_{k+1}|^2-1}{\overline{\lambda_{k+1}}}\right)\langle L_1^{(k)}|\widetilde{v}^{(k)}\rangle\\
&=\frac{\lambda_1\overline{\lambda_{k+1}}-1}{\overline{\lambda_{k+1}}(\lambda_1-\lambda_{k+1})}\langle L_1^{(k)}|\widetilde{v}^{(k)}\rangle
\end{align*}
using~\eqref{eq-Lrecursion}. This yields
\begin{displaymath}
\langle L_1|\widetilde{v}\rangle=\langle L_1^{(N)}|\widetilde{v}^{(N)}\rangle=\widetilde{v}_1\prod_{k=2}^N\left(\frac{\lambda_1\overline{\lambda_{k+1}}-1}{\overline{\lambda_{k+1}}(\lambda_1-\lambda_{k+1})}\right).
\end{displaymath}
Hence, by~\eqref{eq-pert_id2}, we obtain for any $t\in(0,1)$ that
\begin{align*}
\lambda_1'(t)&=\frac{\lambda_1(t)}{t}\langle L_1|v\rangle\langle v|R_1\rangle=\frac{\lambda_1(t)}{t(t^2-1)}\langle L_1|\widetilde{v}\rangle\langle \widetilde{v}|R_1\rangle\\
&=\frac{\lambda_1(t)|\widetilde{v}_1|^2}{t(t^2-1)}\prod_{k=2}^N\left(\frac{\lambda_1(t)\overline{\lambda_k(t)}-1}{\overline{\lambda_k(t)}(\lambda_1(t)-\lambda_k(t))}\right).
\end{align*}
Moreover, replacing $|\widetilde{v}_1|^2$ by 
\begin{displaymath}
|\widetilde{v}_1|^2=(1-|\lambda_1|^2)\prod_{k=2}^N|\lambda_k|^2
\end{displaymath}
using~\cite[Lem.~4]{Dubach_TUE}, we obtain~\eqref{eq-ODEs}, which is a function of $t$ and the eigenvalues only.
\end{proof}

\begin{figure}[b!]
\begin{minipage}[c]{.3\textwidth}
\begin{center}
\includegraphics[width=\textwidth]{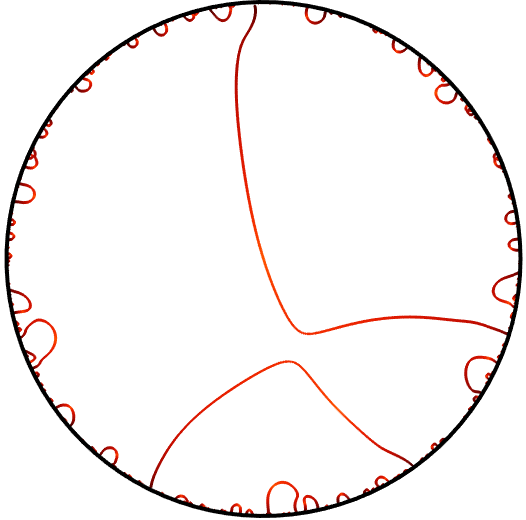} 
\end{center}
\end{minipage}
\begin{minipage}[c]{.65\textwidth}
\captionof{figure}{Trajectories of the eigenvalues of the $UA$ model of size $500\times500$. The evolution of the time parameter $t \in (-1,1)$ is represented by shades of red. On this particular realization, one extreme trajectory gets close to the trajectory of the outlier, and appears to be repelled away from it with high velocity, according to the singularity of the ODE \eqref{eq-ODEs}.}\label{fig2}
\end{minipage}
\end{figure}

\newpage
\subsection{Estimates on the trajectories under the Regularity Assumption}\label{sec_Rouché}

We now state and prove the estimates on the trajectories, under Assumptions \ref{smart_ass} and \ref{kick_ass}, that constitute our main results. The method consists in comparing the function $\W$ to its affine part $s_1$ -- see definitions \eqref{def_om1} and \eqref{def_Ws}. As the eigenvalues are characterized by $\W(z)=\frac{1}{1-t}$ (by Lemma \ref{lem-trajectories}), we denote by $z_t$ the unique solution of the equation $s_1(z) = \frac{1}{1-t}$, that is
\begin{equation}\label{def_zt}
    z_t : = \frac{\sqrt{N}}{\om_1^{(N)}} \frac{t}{1-t},
\end{equation}
so that, in particular, $z_t$ is the expected location of the outlier in the sub-critical regime; but also, $|z_t|$ will appear as a relevant scale parameter in all regimes. \\

In order to deduce the location of eigenvalues from the comparison of $\W$ and $s_1$, we rely on the classical tool of complex analysis known as Rouché's theorem, which we recall:

\begin{theorem}[Rouché's Theorem]
Let $f$ and $g$ be two holomorphic functions on a domain $\Omega\subset\C$ with closed and simple boundary $\partial\Omega$. If $|f-g| < |g|$ on $\partial\Omega$, then $f$ and $g$ have the same number of zeros (counted with multiplicity) inside $\Omega$.
\end{theorem}

The next proposition identifies a domain on which the Rouché inequality holds for the functions at stake here.

\begin{proposition}\label{prop_Rouche_domain_R} Under Assumption \ref{smart_ass}, for every $\eta, \delta>0$, it holds with overwhelming probability that for every $t \in (-1,1)$, the functions 
\begin{equation}\label{functions_ft_gt}
f_t(z) = \W(z) - \frac{1}{1-t} \quad \text{and} \quad g_t(z) = s_1(z) - \frac{1}{1-t}
\end{equation}
verify the inequality $|f_t - g_t|<|g_t|$ on the domain $\mathscr{R}_{t,\eta} \cap \D_\delta$, with
\begin{equation}\label{Rouche_domain}
\mathscr{R}_{t,\eta} := \left\{ z \ : \ \frac{N^{\eta} |z|^2}{1-|z|} < \left|\om_1^{(N)} \right| |z-z_t|  \right\}.
\end{equation}
\end{proposition}

Note that \eqref{Rouche_domain} defines a random domain, depending on $\om_1^{(N)}$, and that Proposition \ref{prop_Rouche_domain_R} only relies on Assumption \ref{smart_ass}.

The proof proceeds by direct computation, unpacking the definitions of the objects and using the relevant assumptions.
\begin{proof} 
For any $\eta> \epsilon >0$ and $z \in \mathscr{R}_{t,\eta}$ the following holds:
\begin{align*}
    |f_t (z) - g_t(z)| & = |\W_2(z)| & \text{(by definition of $f_t, g_t, \W_2$)} \\
    & < \frac{N^{\epsilon} |z|^2}{\sqrt{N} (1-|z|)} & \text{(w.o.p., by Assumption \ref{smart_ass})} \\
    & < N^{\epsilon - \eta} \frac{|\om_1^{(N)}|}{\sqrt{N}} \left| z-z_t \right| & \text{(by definition of $\mathscr{R}_{t,\eta}^{\omega_1}$)} \\
    & = N^{\epsilon - \eta} \left| g_t(z) \right| & \text{(by definition of $g_t$)}
\end{align*}
which yields the expected inequality as $\epsilon < \eta$.
\end{proof}

We can now establish more precise versions of our main result (Theorem \ref{main_theorem_intro}).
\begin{theorem}\label{main_theorem_precise_om1} Under Assumption \ref{smart_ass}, for any $\alpha > \eps > 0$ and $\delta>\frac12$, the following holds with overwhelming probability:
\begin{itemize}
    \item[(i)] For $|t| < |\om_1^{(N)}| N^{-\frac12-\alpha}$,
    exactly one eigenvalue is in $D_1$,  while all other eigenvalues are in $\D \backslash D_{2}$, where $D_1, D_2$ are defined as
    \begin{equation}\label{def_D12}
    D_1 := D\left(z_t, \rho_1 = |z_t|^2 \frac{N^{\epsilon}}{|\om_1|}\right), \qquad
    D_2 := D\left(0, \rho_2 = \left(1+\frac{N^{\epsilon}}{|\om_1|} \right)^{\!\! -1}\right).
    \end{equation}
    \item[(ii)] If $|\om_1^{(N)}| < \frac12 N^{\frac12}$, then for $|t| \in ( 2 |\om_1^{(N)}| N^{-\frac12},1) $, all eigenvalues are in $\D \backslash (D_{3} \cap \D_{\delta})$, where
    \begin{equation}\label{def_D12}
    D_3:=D\left(0,\rho_3 = \left( 1 + \frac{ N^{\eps}}{\left|\om_1\right| |z_t|}\right)^{\!\!-1} \right).
    \end{equation}
\end{itemize}
\end{theorem}

The proof proceeds by a straightforward inspection of the Rouché domain $\mathscr{R}_{t,\eta}$ (with $\eta < \eps$) in the different regimes.

\begin{proof}
For $|t| < |\om_1^{(N)}| N^{-\frac12-\alpha}$, by definition of $z_t$ and using the deterministic bound \eqref{om1_deterministic_bound}, we find
\begin{equation}
|z_t| < N^{-\alpha}, 
\qquad  \frac{|z_t|^2}{|\om_1^{(N)}|} < N^{-\frac12 - 2 \alpha}, 
\qquad |z_t|+\rho_1 < \rho_2 < 1- N^{-\delta},
\end{equation}
so that in particular $D_1 \subset D_2 \subset \D_{\delta}$.
We now check that $\partial D_1, \partial D_2$ are in the domain $\mathscr{R}_{t,\eta}$ with any $\eta < \epsilon$.
\begin{enumerate}
\item For $z \in \partial D_1$, by inspection $1-|z| \asymp 1$, $|z|\asymp |z_t|$. Hence, the left-hand side of the inequality in \eqref{Rouche_domain} is of order $|z_t|^2 N^{\eta}$, while the right-hand side is 
$$|\om_1| |z-z_t| = |z_t|^2 N^{\epsilon} $$
by definition of $D_1$. The inequality follows as we chose $\eta < \eps$. By Rouché's theorem, $D_1$ contains a unique eigenvalue.

\item For $z \in \partial D_2$, the left-hand side of the inequality in~\eqref{Rouche_domain} is $ {N^{\eta} \rho_2^2}{(1-\rho_2)^{-1}} $ and by inspection the right-hand side is of order $|\om_1| \rho_2$, so that the inequality holds if
$$
\frac{ \rho_2}{1-\rho_2} = |\om_1| N^{-\eps}  \ll |\om_1| N^{- \eta}
$$
which, again, is true as $\eta < \epsilon$. By Rouché's theorem, $D_2$ contains a unique eigenvalue, which was already isolated (in $D_1$ by the above argument), so that all other eigenvalues are in $\D \backslash D_2$.
\end{enumerate}

Assuming $|\om_1^{(N)}| < \frac12 N^{\frac12}$, and $|t| > 2|\om_1^{(N)}| N^{-\frac12}$, by definition of $z_t$, we have $|z_t| > 2$, so that $|z-z_t| \asymp |z_t|$. Consider $z \in \partial D_3$. As $\rho_3 <1$, we have $|z|^2< |z|$, and so
$$
\frac{|z|^2}{1-|z|} < \frac{|z|}{1-|z|} = \frac{\rho_3}{1-\rho_3} = N^{-\eps} |\om_1| |z_t|
< N^{-\eta} |\om_1| |z_t|
$$
as $\eta < \eps$, proving that $\partial D_3$ is in the domain $\mathscr{R}_{t,\eta}$. As $s_1(z) - \frac{1}{1-t}$ does not vanish inside $D_3$, we conclude by Rouché's theorem that all eigenvalues are outside $D_3 \cap \D_\delta$.
\end{proof}

Under Assumption~\ref{kick_ass}, the dependence on $\om_1^{(N)}$ can be removed from the domains considered. Note the result now only holds with high probability.

\begin{theorem}\label{main_theorem_precise_no_om1} Under Assumptions \ref{smart_ass} and \ref{kick_ass},
for any $\alpha > \eps>0$, the following holds with high probability:
\begin{itemize}
    \item[(i)] For $|t|<N^{-\frac12-\alpha}$,
    exactly one eigenvalue is in $ D\left(z_t, |z_t|^2 N^{\eps}\right)$,  while all other eigenvalues are in $\D \backslash D\left(0, N^{-\eps}\right)$.
    \item[(ii)] For $|t| > N^{-\frac12+\alpha} $, all eigenvalues are in $\D \backslash D\left(0,1- |z_t|^{-1} N^{\eps}\right)$.
\end{itemize}
\end{theorem}

\begin{proof}
This follows from Theorem \ref{main_theorem_precise_om1} by choosing the parameters carefully: $\delta>0$ from Assumption \ref{smart_ass} large enough and $\eps>0$ from Assumption \ref{kick_ass} small enough.
\end{proof}

\section{Particulars of the CUE case}\label{sec_CUE}
Of particular interest is the situation where $U$ is a Haar-distributed Unitary matrix -- the so-called Circular Unitary Ensemble (CUE). This is the case considered in particular by the papers \cites{Fyodorov2001,Fyodorov_CUE, FyodorovSommers2003, ForresterIpsen,Forrester_Review}. We first prove that our general regularity assumption holds in this case; we then rely on the integrability of the model to give an argument as to the optimality of the main theorem.

\subsection{Proof of the regularity assumption for CUE}
This section is devoted to the proof of the following theorem, which establishes that our main results (Theorems \ref{main_theorem_intro},\ref{main_theorem_precise_om1},\ref{main_theorem_precise_no_om1}) apply in particular to a CUE matrix.
\begin{theorem}\label{thm:assumption_holds}
Assumptions \ref{smart_ass} and \ref{kick_ass} hold for $U$ a $\CUE(N)$ matrix and $v$ a unit vector independent of $U$. More precisely, if $U$ is an $N \times N$ Haar-distributed unitary matrix, then
\begin{equation}\label{assum_2.1_CUE}
    \om_1^{(N)} := \sqrt{N} v^* U^* v 
    \disteq \sqrt{N} \beta_{1,N} \asymp 1,
\end{equation}
and for any $\delta>0$, the stochastic domination
\begin{equation}\label{assum_2.2_CUE}
    \W_2(z) := v^* \frac{(zU^*)^2}{I_N-zU^*} v \prec \frac{|z|^2}{ \sqrt{N} (1-|z|)}
\end{equation}
holds uniformly in $z$ with $1-|z| > N^{-\delta}$.
\end{theorem}

By unitary invariance of CUE matrices, and $v$ being independent from $U$, we can assume $v=e_1$ in all proofs below. Also, throughout the Section, we denote the eigenvalues of $U$ by $\re^{\ri \theta_1}, \dots, \re^{\ri \theta_N}$ and the corresponding eigenvectors by $r_1, \dots, r_N$.

To illustrate the method, we first give a precise estimation of the variance of $\W_2(z)$.

\begin{proposition}\label{exact_var}
For any $z \in \D, N \geq 1$,
\begin{equation}\label{mean_and_var}
\E \W_2(z) = 0, \qquad
\var \W_2 (z) \asymp \frac{|z|^4}{ N (1-|z|)},
\end{equation}
where $\asymp$ means that it is bounded above and below, up to an absolute multiplicative constant.
\end{proposition}

\begin{proof}
For any fixed $z$, the spectral decomposition yields
\begin{equation}\label{spectral_decomposition}
\W_2(z) 
= e_1^* \frac{(zU^{*})^2}{I_N-zU^*} e_1
= 
\sum_{k=1}^N  |r_{k1}|^2
\frac{z^2 \re^{-2\ri\theta_k}}{1-z\re^{-\ri\theta_k}}
=
\sum_{k=1}^N
|r_{k1}|^2
\sum_{\ell \geq 2}
z^l \re^{-\ri \ell \theta_k}
\end{equation}
The expectation being zero follows directly by unitary invariance. For the variance, we write
\begin{align*}
\var \W_2(z)
& = \E
\sum_{k_1, k_2 = 1}^N
|r_{k_1 1}|^2
|r_{k_2 1}|^2
\sum_{\ell_1, \ell_2 \geq 2}
z^{\ell_1} \overline{z}^{\ell_2} \re^{-\ri \ell_1 \theta_{k_1}}
\re^{\ri \ell_2 \theta_{k_2}} \\
& =
\sum_{k_1, k_2 = 1}^N
\E \left(
|r_{k_1 1}|^2
|r_{k_2 1}|^2
\right)
\sum_{\ell \geq 2}
|z|^{2\ell}
\E \left( \re^{-\ri \ell (\theta_{k_1}- \theta_{k_2})} \right)
\end{align*}
where we have used the independence of eigenvectors and eigenvalues to split the expectation, and unitary invariance to cancel the terms with $\ell_1 \neq \ell_2$. The expectation of the eigenvector term is obtained explicitly from Weingarten calculus, namely
\begin{equation}\label{elementary_Wg}
    \E \left(
|r_{k_1 1}|^2
|r_{k_2 1}|^2 \right)
=
\frac{1+\delta_{k_1,k_2}}{N(N+1)}.
\end{equation}
For the eigenvalues we invoke the fact (Proposition 3.11 in \cite{Meckes_book}) that 
\begin{equation}\label{Meckes_magic}
\E |\tr (U^\ell)|^2 = \min (\ell,N)    
\end{equation}
from which it follows in particular that
\begin{equation}\label{two_theta_cancel}
\E \left( \re^{\ri \ell (\theta_1 - \theta_2)} \right) 
= \frac{ \min(\ell,N) - N }{N(N-1)}
= - \frac{(N-\ell)_+}{N(N-1)}.
\end{equation}
Finally, by splitting the sum according to $k_1=k_2$ and $k_1 \neq k_2$, we deduce the following exact expression of the variance of $\W_2$
\begin{align*}
\var \W_2(z)
& =
\sum_{k = 1}^N
\frac{2}{N(N+1)}
\sum_{\ell \geq 2}
|z|^{2\ell}
-
\sum_{k_1 \neq k_2}
\frac{1}{N(N+1)}
\sum_{\ell = 2}^{N}
|z|^{2\ell}
\frac{N-\ell}{N(N-1)} \\
& = \frac{2 |z|^4}{(N+1) (1-|z|^2)} - \frac{1}{N (N+1)} \sum_{\ell =2}^N (N-\ell) |z|^{2\ell} \\
& < \frac{2 |z|^4}{(N+1) (1-|z|^2)},
\end{align*}
which proves the upper bound; but also, from the same identity,
\begin{align*}
\var \W_2(z)
& = \frac{|z|^4}{(N+1) (1-|z|^2)} + \frac{|z|^{2N+2}}{(N+1)(1-|z|^2)} + \frac{1}{N (N+1)} \sum_{\ell =2}^N \ell |z|^{2\ell} \\
& > \frac{|z|^4}{(N+1) (1-|z|^2)},
\end{align*}
which proves the lower bound on the variance.
\end{proof}

To generalize the above argument to all moments, we need the following lemma, which generalizes the overall bound of order $N^{-1}$ in \eqref{two_theta_cancel}.

\begin{lemma}\label{lemma:CUE_bound_1}
For any $m \geq 2$, distinct $k_1, \dots, k_m$ and non-zero integer coefficients $a_1, \dots, a_m$,
\begin{equation}\label{CUE_bound_1}
\left| \E \left( \re^{ \ri \sum_{l=1}^m a_l \theta_{k_l}} \right) \right| 
< m! N^{-m+1}.
\end{equation}
\end{lemma}

\begin{proof}
This follows from classical results on CUE as a determinantal point process (see for instance Proposition 3.7 and Lemma 3.8 in \cite{Meckes_book}). The correlation functions of $\mathrm{CUE}(N)$ are given by
\begin{equation}\label{k-correlation}
   \rho_m (\theta_1, \dots, \theta_m)
   =
   \det\left( K_N( \theta_j, \theta_k) \right)_{j,k=1}^m
\end{equation}
with the kernel
\begin{equation}\label{kernel}
    K_N(x,y) = \sum_{\ell=0}^{N-1} \re^{\ri \ell (x-y)}
\end{equation}
(note in particular that $K_N(x,x)=N$ whereas for $x \neq y$ it is a sum of exponential, not all of which will contribute -- this is the reason behind Lemma 8 in a nutshell)
It follows in particular that
\begin{align*}
\E \left( \re^{ \ri \sum_{l=1}^m a_l \theta_{k_l}} \right)
& =
\frac{(N-m)!}{N!}
\int 
\re^{ \ri \sum_{l=1}^m a_l \theta_{l}}
\det \left( K_N( \theta_j, \theta_k) \right)_{j,k=1}^m \dd \theta_1 \cdots \dd \theta_m \\
& =
\frac{(N-m)!}{N!}
\int 
\re^{ \ri \sum_{l=1}^m a_l \theta_{l}}
\sum_{\sigma \in \mathfrak{S}_m} \epsilon(\sigma)
\prod_{j=1}^m K_N( \theta_j, \theta_{\sigma(j)}) \dd \theta_1 \cdots \dd \theta_m \\
& = \frac{(N-m)!}{N!}
\sum_{\sigma \in \mathfrak{S}_m} 
\sum_{\ell_1, \dots \ell_m =0}^{N-1}
\epsilon(\sigma)
\int 
\re^{ \ri \sum_{l=1}^m a_l \theta_{l}}
\re^{\ri \sum_{j=1}^m \ell_j (\theta_j - \theta_{\sigma(j)})}
\dd \theta_1 \cdots \dd \theta_m
\end{align*}
as each term gives $\pm 1$ if and only if the sums in the exponentials exactly match, and zero otherwise. We bound it by 
\begin{align*}
\left| \E \left( \re^{ \ri \sum_{l=1}^m a_l \theta_{k_l}} \right) \right|
& \leq
\frac{(N-m)!}{N!}
\left| \left\{
\begin{array}{c}
    \sigma \in \mathfrak{S}_m,  \\
     0 \leq \ell_1, \dots, \ell_m \leq N-1 
\end{array}
 \ : \ \sum_{l=1}^m a_l \theta_{l} \equiv \sum_{j=1}^m \ell_j ( \theta_{\sigma(j)} - \theta_j) \right\} \right| \\
& \leq 
\frac{(N-m)!}{N!} m! N,
\end{align*}
where $\equiv$ stands for equality as polynomials in $\theta_1, \dots, \theta_m$. Indeed, there are less than $m!$ choice of $\sigma$, $N$ choices of $\ell_1$, and at most one choice of $(\ell_2, \dots, \ell_m)$ that works once $\sigma$ and $\ell_1$ are fixed\footnote{One could also note that the only choice of $\sigma$ that can possibly work is a large cycle, resulting in the whole sum being of the sign of $\epsilon(\sigma) = (-1)^{m+1}$; but we do not use this particular fact.}. The claim follows.
\end{proof}

\begin{lemma}\label{lemma:CUE_bound_2}
For any $M \geq 1$, indices $ k_1, k_1', \dots, k_M,k_M'$ (not necessarily distinct) and positive integer coefficients $\ell_1, \ell_1' \dots, \ell_M, \ell_M' \geq 1$,
\begin{equation}\label{CUE_bound_2}
\left| \E \left( \re^{ \ri \sum_{j=1}^M \ell_j \theta_{k_j} - \ell_j' \theta_{k_j'}} \right) \right| 
< 
(2M)! N^{M - |\{ k_1, k_1', \dots, k_M, k_M' \}|}.
\end{equation}
\end{lemma}

\begin{proof}
Let $\mathfrak{m} = \mathfrak{m} (\ell_j, \ell_j',k_j, k_j')_{j=1, \dots, M}$ stand for the number of different $\theta_k$ terms with non-zero coefficient in the linear combination $\sum_{j=1}^M \ell_j \theta_{k_j} - \ell_j' \theta_{k_j'}$. The key combinatorial fact will be the following inequality, that holds for any choice of $k,k', \ell, \ell'$:
\begin{equation}\label{combi_fact}
    \left| \{ k_1, k_1', \dots, k_M, k_M' \} \right| \leq \mathfrak{m} + M
    \quad \text{with equality only if} 
    \quad \mathfrak{m} = 0.
\end{equation}
Indeed, the difference between the cardinal $|\{ k_1, k_1', \dots, k_M, k_M' \}|$ and $\mathfrak{m}$ is exactly the number of $\theta_k$ terms that exactly vanish due to a correspondence between coefficients $\ell, \ell'$. Due to the positive sign of these coefficients, this always implies that (at least) one term $\theta_k$ and (at least) one term $\theta_{k'}$ have canceled one another out, which can be the case for at most $M$ distinct coefficients, hence the inequality. Moreover, if such cancellations do concern $M$ distinct coefficients (equality case), then $M$ distinct $\theta_k$ terms vanish with $M$ distinct $\theta_{k'}$ terms, with an exact correspondence between the two sides -- in particular, in this case, the whole linear combination vanishes, and $\mathfrak{m}=0$.

We now consider three possible cases: either $\mathfrak{m}=0$, $\mathfrak{m}=1$ or $\mathfrak{m} \geq 2$.
\begin{enumerate}[label=(\roman*)]
\item If $\mathfrak{m}=0$, the the l.h.s. of \eqref{CUE_bound_2} is $1$ whereas the r.h.s is at least $(2M)!$ by inequality~\eqref{combi_fact}.
\item If $\mathfrak{m}=1$, the l.h.s is $0$ by unitary invariance.
\item If $\mathfrak{m} \geq 2$, the linear combination has $\mathfrak{m}$ distinct terms with non-zero integer coefficients, and Lemma \ref{lemma:CUE_bound_1} gives
$$\left| \E \left( \re^{ \ri \sum_{j=1}^M \ell_j \theta_{k_j} - \ell_j' \theta_{k_j'}} \right) \right| 
<
\mathfrak{m}! \ N^{-\mathfrak{m}+1}
\leq (2M)! \ N^{-\mathfrak{m}+1}.
$$
Moreover, as $\mathfrak{m} \neq 0$ we are not in the equality case of \eqref{combi_fact} and so
$$
- \mathfrak{m} + 1 \leq M- \left| \{ k_1, k_1', \dots, k_M, k_M' \} \right|.
$$
\end{enumerate}
The claim follows.
\end{proof}

\begin{proposition}\label{moment_bound}
For any $M \geq 1$, $z \in \D$, for large enough $N$,
\begin{equation}
    \E |\W_2 (z)|^{2M} < \frac{C_M |z|^{4M}}{N^M (1-|z|)^{2M-1}}
\end{equation}
where $C_M$ is a constant depending only on $M$ and not on $N,z$.
\end{proposition}

\begin{proof}
We start with the spectral decomposition, expand, and split the expectation as in the proof of Proposition \ref{exact_var}. This yields
\begin{align*}
\E |\W_2 (z)|^{2M}
& = \hspace{-.2in}
\sum_{k_1,k_1', \dots, k_M,k_M'=1}^N
\hspace{-.1in}
\E \left(
\prod_{j=1}^M |r_{1,k_j}|^2 |r_{1,k_j'}|^2 
\right)
\sum_{\ell_1,\ell_1', \dots, \ell_M, \ell_M' \geq 2}
\hspace{-.2in}
z^{\sum_{j=1}^{M} \ell_j} \ovr{z}^{\sum_{j=1}^{M} \ell_j'}
\E \left( \re^{\sum_{j=1}^M \ell_j \theta_{k_j} - \ell_j' \theta_{k_j'}} \right) \\
\leq
|z|^{4M} & \hspace{-.1in}
\sum_{
k_1,k_1', \dots, k_M,k_M'=1
}^N
\underbrace{\E \left(
\prod_{j=1}^M |r_{1,k_j}|^2 |r_{1,k_j'}|^2 
\right)}_{\mathrm{(I)}}
\sum_{
\substack{
    \ell_1,\ell_1', \dots, \ell_M, \ell_M' \geq 0 \\
    \sum \ell_j = \sum \ell_j'}
}
\hspace{-.1in}
|z|^{2\sum_{j=1}^M \ell_j}
\underbrace{
\left| \E \left( \re^{\sum_{j=1}^M \ell_j \theta_{k_j} - \ell_j' \theta_{k_j'}} \right) \right|}_{\mathrm{(II)}},
\end{align*}
where we used that, by unitary invariance, only the terms with $\sum_{j=1}^M \ell_j = \sum_{j=1}^M \ell_j'$ contribute to the second sum. We now proceed in successive bounds. \\

The eigenvector term (I) is bounded by a general Weingarten bound
\begin{equation}\label{crude_bound_1}
\E \left(
\prod_{j=1}^M |r_{1,k_j}|^2 |r_{1,k_j'}|^2 \right)
< C_M N^{-2M}
\end{equation}
with a constant $C_M$ that depends only on $M$. This can be obtained for instance from Theorem 1.1 in \cite{CollinsMatsumoto}, provided $N > 2\sqrt{3}M^{7/4}$ (note that this requirement is the only reason for the provision of $N$ being "large enough" in the statement). Indeed, this theorem gives a uniform bound on all Weingarten function $\Wg (\sigma \tau^{-1}, N )$ with $(\sigma, \tau) \in \mathfrak{S}_M^2$. By the fundamental theorem of Weingarten calculus, one can write the expectation (I) as a sum of such Weingarten functions with admissible pairs $(\sigma, \tau)$. Moreover, this sum has a maximum of $(M!)^2$ terms -- a number which only depends on $M$ and is therefore absorbed into the constant $C_M$.  \\

The eigenvalue term (II) is directly bounded by Lemma \ref{lemma:CUE_bound_2}. The factor $(2M)!$ is absorbed into the constant $C_M$, so that together, these first two bounds yield
\begin{align*}
\E |\W_2 (z)|^{2M}
& <
C_M
|z|^{4M}
\hspace{-.2in}
\sum_{k_1,k_1', \dots, k_M,k_M'=1}^N
\hspace{-.2in}
N^{-2M}
\sum_{
\substack{
    \ell_1,\ell_1', \dots, \ell_M, \ell_M' \geq 0 \\
    \sum \ell_j = \sum \ell_j'}
}
N^{M - |\{ k_1,k_1', \dots, k_M,k_M' \}|} |z|^{2\sum_j \ell_j} \\
&  = \frac{C_M |z|^{4M}}{N^M}
\hspace{-.2in}
\underbrace{
\sum_{k_1,k_1', \dots, k_M,k_M'=1}^N
\hspace{-.2in}
N^{ - |\{ k_1,k_1', \dots, k_M,k_M' \}|} 
}_{\mathrm{(III)}}
\underbrace{
\sum_{\ell_1,\ell_1', \dots, \ell_M, \ell_M' \geq 0}
    \mathbf{1}_{\sum \ell_j = \sum \ell_j'}
    |z|^{2\sum_{j=1}^M \ell_j} }_{\mathrm{(IV)}}
\end{align*}
The sum (III) can be written with a first sum over configurations of $k_1, k_1', \dots k_M, k_M'$, i.e. partitions $\pi$ of $[\![1,2M]\!]$, and then a sum over choices of indices coherent with $\pi$:
\begin{equation}\label{partition_of_ks}
    \mathrm{(III)} =
    \sum_{\pi \in \mathscr{P}([\![1,2M]\!])} \sum_{ (k_j,k_j')_{j=1}^{M} \sim \pi } N^{ - |\{ k_1,k_1', \dots, k_M,k_M' \}|} ,
\end{equation}
where by $(a_j)_{j} \sim \pi$ we mean that $a_i = a_j$ if and only if $i,j$ are in the same block of $\pi$. For a given partition $\pi$, the number of possible choices of $k_j, k_j'$ is exactly
$$
N (N-1) \cdots (N-|\pi|+1)
\leq 
N^{|\pi|}
= N^{|\{ k_1,k_1', \dots, k_M,k_M' \}|}
$$
so that (III) is bounded by $|\mathscr{P}( [\![1,2M]\!])|$. Note that this constant depends only on $M$ and we can absorb it into $C_M$. \\

The last term (IV) can be computed exactly. Recall that the number of partitions of an integer $L \geq 0$ in $M \geq 1$ parts is given by
\begin{equation}\label{L_in_M_parts}
    \left| \left\{
    \ell_1, \dots, \ell_M \geq 0 \ : \ 
    \sum_{i=1}^M \ell_i = L
    \right\} \right|
    =
    \binom{L+M-1}{M-1}.
\end{equation}
The expression (IV) can hence be identified as a hypergeometric function of $|z|^2$, i.e.,
\begin{equation}\label{link_with_2F1}
\sum_{\ell_1,\ell_1', \dots, \ell_M, \ell_M' \geq 0}
\hspace{-.1in} |z|^{2\sum_j \ell_j}
\mathbf{1}_{\sum_j \ell_j = \sum_j \ell_j'}
=
\sum_{L \geq 0}
\binom{L+M-1}{M-1}^2
|z|^{2L}
=
{}_2 F_1 \left( M,M;1;|z|^2 \right).
\end{equation}
We remind the reader of the definition of the hypergeometric functions
\begin{equation}\label{def_2F1}
    {}_2 F_1 \left( a,b ; c; z \right) = \sum_{n \geq 0} \frac{(a)_n (b)_n}{(c)_n} \frac{z^n}{n!},
\end{equation}
where $(\cdot)_n$ is the rising Pochhammer symbol, defined for any $q \in \C$ by
\begin{equation}\label{def_pochhammer}
 (q)_n = 
    \left\{
    \begin{array}{rl}
        1 & \text{ if } n=0 \\
        q(q+1) \cdots (q+n-1) & \text{ if } n \geq 1 .
    \end{array}
    \right.
\end{equation}
One of the most celebrated identities of hypergeometric functions is Euler's identity
\begin{equation}\label{Euler_identity}
    {}_2 F_1 \left( a,b ; c; z \right) = 
    (1-z)^{c-a-b}
    {}_2 F_1 \left( c-a , c-b ; c; z \right).
\end{equation}
Applying~\eqref{Euler_identity} to the parameters from \eqref{link_with_2F1} yields
\begin{align*}\label{Euler_identity_applied}
    {}_2 F_1 \left( M,M;1;|z|^2 \right) & = 
    (1-|z|^2)^{1-2M}
    {}_2 F_1 \left( 1-M , 1-M ; 1; |z|^2 \right) \\
    & = \frac{1}{(1-|z|^2)^{2M-1}} \sum_{n=0}^{M-1} \frac{((1-M)_n)^2}{n!^2} |z|^{2n}
\end{align*}
where the last equality follows from the definitions \eqref{def_2F1} and \eqref{def_pochhammer}. The last term is a polynomial with positive terms and hence bounded by a constant depending only on $M$ (for instance, its value at $|z|=1$). This concludes the proof. \end{proof}
We can now give the proof of Theorem \ref{thm:assumption_holds}.

\begin{proof}[Proof of Theorem \ref{thm:assumption_holds}]
For CUE, it is well known that $\om_1 = U_{11}^* \disteq \re^{\ri \theta} \sqrt{\beta_{1,N}}$ with uniform $\theta$ and an independent beta distribution (see, e.g.,~\cite[Ch.~2]{Meckes_book} or~\cite{BHNY}). This implies~\eqref{assum_2.1_CUE}. \\

It remains to establish \eqref{assum_2.2_CUE}. Observing that the desired stochastic domination bound trivially holds at the origin, let $z\neq0$ and consider the function $\smash{\widetilde{\W}_2(z):=z^{-2}\mathscr{W}_2(z)}$. Fix $\ell>2\delta$ and let $\mathds{L}\subseteq\D$ be a lattice with $|\mathds{L}|\simeq N^\ell$ such that for every $z\in \D_{\delta}=D(0,1-N^{-\delta})$ there is some $z_0\in\mathds{L}$ with $|z-z_0|\leq C_{\mathds{L}}N^{-\ell/2}$. We start by establishing a stochastic domination bound for $\smash{\widetilde{\W}_2}$ that holds point-wise at $z_0 \in \mathds{L}\cap \D_{\delta}$. Fix $\eps,A>0$ and let $M\in\N$ such that $2\eps M>A+\ell$. Using Markov's inequality and the moment bound in Proposition~\ref{moment_bound} yields
\begin{equation}\label{eq-sdbound1}
\PP\left(\widetilde{\W}_2(z_0)>N^\eps\frac{1}{\sqrt{N}(1-|z_0|)}\right)\leq\frac{C_M}{N^{2M\eps}}\leq N^{-A-\ell}.
\end{equation}
In particular, $\smash{\W_2(z_0)\prec\frac{|z_0|^2}{\sqrt{N}(1-|z_0|)}}$ for any lattice point $z_0$ with $|z_0|<1-N^{-\delta}$. By a union bound,~\eqref{eq-sdbound1} extends to
\begin{displaymath}
\PP \left( \exists z_0\in\mathds{L}\cap \D_{\delta}: \widetilde{\W}_2(z_0)>N^\eps\frac{1}{\sqrt{N}(1-|z_0|)} \right)< N^{-A}.
\end{displaymath}
Next, let $z\in \D_{\delta}$. Recalling that we may assume $v=e_1$, compute
\begin{displaymath}
\widetilde{\W}_2'(z)=\left(\frac{2zU^*(I_N-zU^*)-z^2(U^*)^3}{z^2(I_N-zU^*)^2}-\frac{(zU^*)^2}{z^3(I_N-zU^*)}\right)_{11}=\left(\frac{-(U^*)^3}{(I_N-zU^*)^2}\right)_{11}
\end{displaymath}
We estimate the derivative by
\begin{displaymath}
|\widetilde{\W}_2'(z)|\leq \left\|\frac{-(U^*)^3}{(I_N-zU^*)^2}\right\|\leq \frac{1}{(1-|z|)^2},
\end{displaymath}
where the last inequality follows from the spectral theorem and the fact that all eigenvalues of $U$ lie on the unit circle. Applying the mean value theorem for the analytic function $\widetilde{\W}_2(z)$ yields
\begin{displaymath}
\left|\widetilde{\W}_2(z)-\widetilde{\W}_2(z_0)\right| \leq |z-z_0| \max \left\{ \left|\widetilde{\W}_2'(\xi)\right|\ :\ \xi \in [z_0,z] \right\} \leq C_{\mathds{L}}N^{-\ell/2} \frac{1}{(1-|z|)^2} \leq C_{\mathds{L}}N^{-\ell/2+\delta} \frac{1}{1-|z|}
\end{displaymath}
where we w.l.o.g. picked $z_0\in\mathds{L}\cap\D_{\delta}$ that is closer to the center of the disc than $z$. Hence, the stochastic domination bound extends to all of $\D_{\delta}$ and we obtain
\begin{displaymath}
\widetilde{\W}_2(z)\prec \frac{1}{\sqrt{N}(1-|z|)}
\end{displaymath}
uniformly for $z\in D(0,1-\delta)\setminus\{0\}$. As multiplying both sides with $|z|^2$ does not change the stochastic domination bound,~\eqref{assum_2.2_CUE} follows.
\end{proof}

\subsection{Optimality of the critical timescale for CUE}
The estimates derived above from Assumption \ref{smart_ass} (see Proposition \ref{prop_Rouche_domain_R}) do not give sufficient precision at the critical timescale.
What holds in the CUE case (and is expected more generally) is that, typically, for $t= \mu N^{-\frac12}$, several eigenvalues are at an order one distance from both the origin and the unit circle, implying in particular that the outlier is not strongly separated at the critical timescale.
\begin{theorem}
If $U$ is Haar distributed, and $v$ is independent from $U$, there is no strongly separated outlier at time $t=\alpha N^{-\frac12}$.
\end{theorem}
We give a short proof inspired by \cite[Sec.~3.3]{FyodorovSommers2003} that relies on an exact formula from \cite{Fyodorov2001} and a Poissonization trick. A more complete treatment of these statistics is provided in \cite{Fyodorov_CUE}.
\begin{proof}
Let $t \in [0,1]$. It is known from \cite{Fyodorov2001} 
that the joint density of the eigenvalues of the $UA$ model at time $t$ is proportional to\footnote{Note that this density would be easy to handle, if not for the delta function. This condition is the symptom of a dimensional mismatch, which we resolve by taking a random $t$ (`Poissonization'), thus adding a degree of freedom that we can later remove (`de-Poissonization') to prove the claim.}
\begin{equation}\label{joint_density_at_t}
(1-t^2)^{1-N}\delta\left(t^2-\prod_{l=1}^N|z_l|^2\right)\prod_{1\leq j<k\leq N}|z_j-z_k|^2
\end{equation}
with a normalizing constant that does not depend on $t$. We now assume the parameter $t$ to be random such that $t^2$ has distribution $F(x)\dd x$ for some suitable density $F$. The eigenvalues of the $UA$ model at this random time $t$ have the joint density
\begin{equation}\label{poissonized_density_F}
\left(1-\prod_{l=1}^N|z_l|^2 \right)^{1-N} F\left( \prod_{l=1}^N|z_l|^2 \right)\prod_{1\leq j<k\leq N}|z_j-z_k|^2.
\end{equation}
For any $k \geq 1$, we choose $F(x) = x^{k-1}(1-x)^{N-1}$, which is to say that $t^2$ follows a $\beta_{k,N}$ distribution. The joint density of eigenvalues now simplifies to
\begin{equation}\label{poissonized_density_beta}
\prod |z_l|^{2(k-1)}\prod_{1\leq j<k\leq N}|z_j-z_k|^2,
\end{equation}
i.e., they form a determinantal point process with radial symmetry. 
In this setting, Kostlan's theorem gives a particularly nice description of the distribution of the squared norm of the radii, namely
\begin{equation}\label{Kostlan}
\{|\la_1|^2,\dots,|\la_N|^2\}\overset{d}{=}\{\beta_{k,1},\dots,\beta_{k+N-1,1}\}
\end{equation}
with the $\beta$ variables on the right-hand side being independent (see \cite{Kostlan} for Kostlan's original paper in the complex Ginibre case, and \cites{Dubach_Powers,HKPV} for generalizations of this counter-intuitive property and various comments).

The relevant probabilities can now be computed exactly. Assuming $k=1$ for simplicity, the random time parameter is of order $N^{-\frac12}$. More precisely, for any $b>a>0$ we have

\begin{align*}
    \PP \left( t \in \left(\frac{a}{\sqrt{N}} , \frac{b}{\sqrt{N}} \right) \right)
    & = \PP \left(\beta_{1,N} \in \left( \frac{a^2}{N}, \frac{b^2}{N} \right) \right)
    = N \int_{a^2/N}^{b^2/N} (1-x)^{N-1} \dd x \\
    & = \left( 1 - \frac{a^2}{N} \right)^N
    - \left( 1 - \frac{b^2}{N} \right)^N
    \rightarrow \re^{-a^2} - \re^{-b^2} > 0,
\end{align*}
using that $t^2$ is $\beta_{1,N}$-distributed. The probability of having either none or several eigenvalues at some distance from the origin or the unit circle can be given explicitly using~\eqref{Kostlan}. For any $a \in (0,1)$, it follows that
$$
\PP \left( \forall j, |\la_j|^2 > a \right)
 = \prod_{j =1}^N \PP \left( \beta_{j,1} > a \right)
 = \prod_{j=1}^N \int_{a}^{1} j x^{j-1} \dd x 
 = \prod_{j=1}^N (1 - a^j) > \prod_{j = 1}^{\infty} (1 - a^j).
$$
This last product is the inverse of the generating series of partitions of integers, which is known to be convergent for $|a|<1$. This implies that the above probability is non-zero (i.e., there is a non-zero probability that a disk of order one is empty of eigenvalues), and that this probability tends to $1$ when $a \rightarrow 0$. 

Let $E_{2,b}$ be the event that more than one eigenvalue is in $D(0,b)$. Similar to the above computations, we deduce the lower bound
$$
\PP \left( E_{2,b} \right)
> \PP \left( \beta_{1,1} < b^2, \beta_{2,1} < b^2 \right)
= b^6 \xrightarrow[b \rightarrow 1^-]{} 1.$$
The claim now follows by contradiction. Assume there is a strongly separated outlier for all times $t_0 > \mu N^{-1/2}$. Denote by $p_{\mu}>0$ the probability of the random time $t$ being in this regime, recalling that $t^2$ follows a $\beta_{1,N}$ distribution. In the given setting, we can choose a radius $a>0$ such that the disk $D(0,a)$ is free of eigenvalue with probability larger than $1-\frac{p_\mu}{2}$. These two observations result in a positive probability of both events (empty disk and  $t > \mu N^{-\frac12}$) occurring simultaneously. Hence, the strong separation must hold with $\alpha_1\geq0$.

Similarly, we can choose a radius $b>0$ such that there is a probability larger than $1-\frac{p_\mu}{2}$ of having two eigenvalues in $D(0,b)$. By a similar argument, the strong separation must hold with $\alpha_2 \leq 0$. This is a contradiction, as the definition of strong separation requires that $\alpha_1 < \alpha_2$.
\end{proof}

\section*{Appendix: on the regularity assumption}

We would like to close this work with a few remarks about the regularity assumption for $\W_2$ (Assumption \ref{smart_ass}) and how it differs from the isotropic local law used in \cite{DubachErdos}. \\

First, note that this assumption holds for models beyond the CUE case and does not rely on level repulsion. Indeed, the argument provided for CUE above applies, for instance, to any unitary invariant matrix (i.e. $\smash{U \disteq QUQ^*}$ for any unitary $Q$) with eigenvalues distributed according to a generic $\alpha$-determinantal processes with $\alpha \in [-1,1]$ (see \cites{alphadet1,alphadet2} for a definition and history of these processes). In particular, we obtain an analog of Lemma~\ref{lemma:CUE_bound_1} with an $\alpha$-determinant, provided only that $|\alpha|\leq 1$. This covers the special case of i.i.d. points ($\alpha = 0$) and that of a permanental point process ($\alpha=1$), for which repulsion is replaced by some amount of attraction between points. A natural question, then, would be to find optimal conditions under which the regularity assumption holds. \\

Secondly, we contrast Assumption \ref{smart_ass} with the stochastic domination bound for CUE:
\begin{equation}\label{noniso_law}
\tr \frac{(zU^*)^2}{1-zU^*}
\prec 
\frac{|z|^2}{N (1-|z|)^2}
\end{equation}
which holds uniformly in $\D_\delta$. The proof of~\eqref{noniso_law} is straightforward from Johansson's lemma, which we briefly recall for the convenience of the reader.
\begin{lemma}[Lemma~2.9 in~\cite{Johansson}]\label{lem-Johansson}
Let $U$ be a CUE matrix and $f$ be a real function defined on the unit disk. Further, define
\begin{displaymath}
\|f \|_H^2 = \sum_{k \in \mathbb{Z}} |k| |\hat{f}_k|^2,\qquad \hat{f}_k = \frac{1}{2\pi} \int_0^{2\pi} f(\re^{\ri \theta})  \re^{-\ri k\theta} \dd \theta.
\end{displaymath}
If $\| f \|_H<\infty$, then
\begin{displaymath}
\E\left( \re^{\tr f(U)}\right) \leq \re^{N \hat{f}_0 + \frac12 \| f \|_H^2}.
\end{displaymath}
\end{lemma}
Note that~\eqref{noniso_law} may be interpreted (up to normalization) as an averaged local law for the same object as Assumption \ref{smart_ass}. The bound \eqref{noniso_law} is more precise and relies on level repulsion (but not on any assumption concerning the eigenvectors, which do not play any role here). It also follows from Lemma \ref{lem-Johansson} that, in the CUE case, the bound in the regularity assumption can be strengthened to $|z|^2 (N(1-|z|))^{-1/2}$, an improvement that relies on level repulsion and so holds with less generality. Moreover, this does not affect the estimates related to the emergence of an outlier, only the bound on all eigenvalues close to the unit circle in the supercritical regimes.

\end{document}